\documentclass[12pt]{amsart}

\usepackage[margin=1.5in]{geometry}

\usepackage{amsmath}
\usepackage{amsfonts}
\usepackage{amssymb}
\usepackage{amsthm}
\usepackage{enumerate}
\usepackage{graphicx}
\usepackage{hyperref}
\usepackage{enumitem}

\input xy
\xyoption{all}

\theoremstyle{plain}
\newtheorem{theorem}{Theorem}[section]
\newtheorem{corollary}[theorem]{Corollary}
\newtheorem{lemma}[theorem]{Lemma}
\newtheorem{proposition}[theorem]{Proposition}
\theoremstyle{definition}
\newtheorem{definition}[theorem]{Definition}
\newtheorem{example}[theorem]{Example}
\newtheorem{rem}[theorem]{Remark}

\DeclareMathOperator{\End}{End}
\DeclareMathOperator{\Hom}{Hom}
\DeclareMathOperator{\aut}{Aut}

\newcommand{\cal}[1]{\mathcal{#1}}
\newcommand{\bb}[1]{\mathbb{#1}}

\newcommand{\NS}{\mathrm{NS}}
\newcommand{\Pic}{\mathrm{Pic}}
\def\Z{{\mathbb{Z}}}
\def\Q{{\mathbb{Q}}}

\def\C{{\mathbb{C}}}

\makeatletter
\def\ps@pprintTitle{%
  \let\@oddhead\@empty
  \let\@evenhead\@empty
  \let\@oddfoot\@empty
  \let\@evenfoot\@oddfoot
}
\makeatother

\usepackage{marvosym}

\title{Smooth quotients of principally polarized abelian varieties}
\author{Robert Auffarth}
\address{R. Auffarth \\Departamento de Matem\'aticas, Facultad de
Ciencias, Universidad de Chile\\ Las Palmeras 3425, \~Nu\~noa, Santiago, Chile}
\email{rfauffar@uchile.cl}
\author{Giancarlo Lucchini Arteche}
\address{G. Lucchini Arteche \\Departamento de Matem\'aticas, Facultad de
Ciencias, Universidad de Chile\\ Las Palmeras 3425, \~Nu\~noa, Santiago, Chile}
\email{luco@uchile.cl}

\thanks{The first author was partially supported by Fondecyt Grant 11180965. The second author was partially supported by Fondecyt Grant 11170016 and PAI Grant 79170034.}


\begin{document}

\begin{abstract}
We give an explicit characterization of all principally polarized abelian varieties $(A,\Theta)$ such that there is a finite subgroup $G\subseteq\mathrm{Aut}(A,\Theta)$ such that the quotient variety $A/G$ is smooth. We also give a complete classification of smooth quotients of Jacobians of curves.\\

\end{abstract}

\maketitle

\section{Introduction}

Let $(A,\Theta)$ be a complex principally polarized abelian variety and denote by $\aut(A,\Theta)$ the group of automorphisms of $A$ that fix the origin as well as the numerical class of $\Theta$. The purpose of this article is to give an explicit character\-ization of all principally polarized abelian varieties admitting a smooth quotient, that is, varieties $(A,\Theta)$ such that there exists a non-trivial subgroup $G\subseteq\aut(A,\Theta)$ such that $A/G$ is a smooth variety. We recall the well-known fact that $\aut(A,\Theta)$ is a finite group (see \cite[Corollary 5.1.9]{BL}) and so we are only dealing with quotients by finite groups. 

This article can be seen as a continuation of the article \cite{ALA} by the same authors where abelian varieties that have a finite group of automorphisms with smooth quotient are characterized. The addition of a polarization to the study of smooth quotients throws a non-trivial extra ingredient into the mix and will allow us to study the problem in a moduli-theoretic context.

Our first main result (Theorem \ref{main thm}) states the following:

\begin{theorem}
Let $(A,\Theta)$ be a principally polarized abelian variety. Then there exists a non-trivial subgroup $G\subseteq\aut(A,\Theta)$ such that $A/G$ is smooth if and only if $(A,\Theta)$ is the product of principally polarized elliptic curves and a principally polarized abelian variety that comes from what we call the \textit{standard construction}.
\end{theorem}

The basic blocks of our standard construction (see Section \ref{standard} for details) are self-products of an elliptic curve with a natural action of the symmetric group $S_{g+1}$ that comes from its standard representation, and are equipped with a non-principal $S_{g+1}$-invariant polarization $\Xi_g$. The standard construction takes a product of such varieties, as well as an auxiliary polarized abelian variety, and yields a principally polarized abelian variety that is isogenous to the whole product and admits a smooth quotient.

We are able to give a moduli-theoretic interpretation of the situation as well (see Section \ref{sec moduli}). For $g_1,\ldots,g_r,n\in\mathbb{Z}_{\geq0}$ we define a morphism
\[\Phi_{g_1,\ldots,g_s,n}:\left(\prod_{j=1}^s\cal X(g_i+1)\right)\times\mathcal{A}_n^D\to\mathcal{A}_g,\]
where $g=n+z+\sum_{i=1}^sg_i$, $z$ is the number of $g_i$'s equal to 0, $\mathcal{X}(m)$ is the modular curve $\mathbb{H}/\Gamma(m)$ and $\mathcal{A}_n^D$ is the moduli space of polarized abelian varieties of dimension $n$ and type $D$ along with some extra data (where $D$ depends on the $g_i$), whose image consists of principally polarized abelian varieties admitting a smooth quotient. Theorem \ref{moduli} rephrases our main theorem by stating that a principally polarized abelian variety admits a smooth quotient if and only if it lies in the image of one of these morphisms. When $g_i>0$ for all $i$, the image of $\Phi_{g_1,\ldots,g_r,n}$ corresponds to varieties that come from the standard construction described above, and we prove (Proposition \ref{prop standard irred}) that for a very general element in the image of this morphism, its theta divisor is irreducible.\\

Since Jacobians are a special case of irreducible principally polarized abelian varieties, it is then natural to ask what Jacobians have smooth quotients. With a little more work we are able to classify all smooth quotients of Jacobians by groups of automorphisms that come from automorphisms of the curve in question. The classification we obtain for Jacobians is the following: 

\begin{theorem}\label{thm jacob intro}
Let $X$ be a smooth projective curve of genus $g$ and let $G$ be a (non-trivial) group of automorphisms of $X$. Then $J_X/G$ is smooth if and only if one of the following holds:
\begin{enumerate}
\item $g\leq 1$;
\item $g=2$, $G\cong\mathbb{Z}/2\Z$  and $X\to X/G$ ramifies at two points.
\item $g=3$, $G\cong\mathbb{Z}/2\Z$  and $X\to X/G$ is \'etale.
\end{enumerate}
\end{theorem}

Throughout this paper, a \textit{polarization} will be used interchangeably as an ample divisor, an ample line bundle, or an ample numerical class. This should not produce confusion. In the case of a principal polarization, we will say that it is \textit{irreducible} if any effective divisor inducing the polarization is irreducible. We will denote numerical equivalence by $\equiv$.

\subsection*{Acknowledgements} We would like to thank an anonymous referee for an argument that shortened our original proof of Theorem \ref{thm jacob intro}, as well as a second anonymous referee for his or her thorough reading of our article and his or her comments.

\section{Preliminary results: smooth quotients of abelian varieties}\label{sec prelim}

The following results are consequences of the main results of \cite{ALA} and we will be using them in what follows.

\begin{theorem}\label{thm classification}
Let $A$ be an abelian variety of dimension $g$, and let $G$ be a (non-trivial) finite group of automorphisms of $A$ that fix the origin. Then the following conditions are equivalent:
\begin{itemize}
\item[(1)] $A/G$ is smooth and the analytic representation of $G$ is irreducible.
\item[(2)] $A/G$ is smooth of Picard number 1.
\item[(3)] $A/G\cong\mathbb{P}^g$.
\item[(4)] There exists an elliptic curve $E$ such that $A\cong E^g$ (as a variety, not including the polarization) and $(A,G)$ satisfies exactly one of the following:
\begin{enumerate}[label=(\alph*)]
\item $G\cong C^g\rtimes S_g$ where $C$ is a (cyclic) subgroup of automorphisms of $E$ of order $\geq 2$ that fix the origin; here the action of $C^g$ is coordinatewise and $S_g$ permutes the coordinates.\label{ex1}
\item $G\cong S_{g+1}$ and acts on 
\[A\cong\{(x_1,\ldots,x_{g+1})\in E^{g+1}:x_1+\cdots+x_{g+1}=0\},\]
by permutations.\label{ex2}
\item $g=2$, $E=\C/\Z[i]$ and $G$ is the order 16 subgroup of $\mathrm{GL}_2(\Z[i])$ generated by:
\[\left\{\begin{pmatrix} -1 & 1+i \\ 0 & 1\end{pmatrix}\right.,\, \begin{pmatrix} -i & i-1 \\ 0 & i\end{pmatrix},\, \left.\begin{pmatrix} -1 & 0 \\ i-1 & 1\end{pmatrix} \right\},\]
acting on $A$ in the obvious way.\label{ex3}
\end{enumerate}
\end{itemize}
\end{theorem}

To see the proof of this theorem see \cite[Theorem 1.1]{ALA} and \cite[Theorem 1.1]{ALAQ}. We will refer to the three cases appearing in point (4) of this theorem as Example \ref{ex1}, Example \ref{ex2} and Example \ref{ex3} respectively.

The next two results are proved in \cite[Theorem 1.3]{ALA} and \cite[Prop. 2.9]{ALA}.

\begin{theorem}\label{thm red irred}
Let $A$ be an abelian variety of dimension $g$, and let $G$ be a (non-trivial) finite group of automorphisms of $A$ that fix the origin. Assume that $A/G$ is smooth and $\dim(A^G)=0$. Then $G=\prod_{i=1}^rG_i$, $A=\prod_{i=1}^rA_i$, $G_i$ acts trivially on $A_j$ for $i\neq j$ and irreducibly on $A_i$ and $A_i/G_i$ is smooth for all $1\leq i\leq r$. In particular,
\[A/G\cong A_1/G_1\times\cdots\times A_r/G_r.\]
\end{theorem}

\begin{proposition}\label{prop And smooth}
Let $A$ be an abelian variety of dimension $g$, and let $G$ be a (non-trivial) finite group of automorphisms of $A$ that fix the origin. Let $A_0$ be the connected component of $A^G$ containing 0 and let $P_G$ be its complementary abelian subvariety with respect to a $G$-invariant polarization. Then there exists a fibration $A/G\to A_0/(A_0\cap P_G)$ with fibers isomorphic to $P_G/G$. Moreover, $A/G$ is smooth if and only if $P_G /G$ is smooth.
\end{proposition}

\section{Smooth quotients of principally polarized abelian varieties}

In this section, we give a full classification of smooth quotients of principally polarized abelian varieties.

We start with the case where $\dim(A^G)=0$, which is a direct application of Theorems \ref{thm classification} and \ref{thm red irred}. Here we only get direct products of copies of Example \ref{ex1}. Next, we present our standard construction, which uses copies of Example \ref{ex2} in order to obtain new smooth quotients of principally polarized abelian varieties. Finally, we show that in general, every principally polarized abelian variety admitting a smooth quotient is obtained as a direct product of these two cases. We conclude with a moduli-theoretic version of these results.

\subsection{The case $\dim(A^G)=0$}\label{sec dim A^G 0}
Let $E$ be an elliptic curve, and on $E^g$ consider the natural principal polarization
\begin{equation}\label{eqn Theta0}
\Theta_{g}:=[0]\boxtimes\cdots\boxtimes[0],
\end{equation}
where $[0]$ denotes the divisor on $E$ consisting of the origin. Note that in particular $\Theta_1=[0]$.

\begin{proposition}\label{thm ppav irred}
Let $(A,\Theta)$ be a principally polarized abelian variety of dimension $g$ and let $G\subseteq \aut(A,\Theta)$ be a (non-trivial) group of automorphisms of $A$. Assume that $A/G$ is smooth and that the analytic representation of $G$ is irreducible. Then $(A,G)$ is as in Example \ref{ex1} and $\Theta\equiv\Theta_{g}$.
\end{proposition}

\begin{proof}
By Theorem \ref{thm classification}, the pair $(A,G)$ must be as in one of the Examples \ref{ex1}, \ref{ex2}, \ref{ex3}. In particular, $A/G\cong\bb P^g$ and thus $\text{NS}(A)^G=\Z\Theta$. Therefore, if we denote by $\pi$ the morphism $A\to\bb P^g$, there exists $m\in\mathbb{Z}$ such that $\pi^*\mathcal{O}_{\mathbb{P}^g}(1)\equiv m\Theta.$ Then, by taking the self-intersection, we obtain $|G|=m^gg!$.

If $g>1$, this is only possible if $G\cong C^g\rtimes S_g$ with $m=|C|$ and therefore the pair $(A,G)$ is as in Example \ref{ex1}. If $g=1$ then this is trivially true, since in this case Example \ref{ex1} coincides with Example \ref{ex2}. Now, a direct computation in Example \ref{ex1} tells us that $\pi^*\mathcal{O}_{\mathbb{P}^g}(1)$ is indeed $m([0]\boxtimes\cdots\boxtimes[0])$. Hence $\Theta\equiv\Theta_{g}$.
\end{proof}

\begin{corollary}\label{thm ppav inv of dim 0}
Let $(A,\Theta)$ be a principally polarized abelian variety of dimension $g$ and let $G\leq \aut(A,\Theta)$ be a (non-trivial) group of automorphisms of $A$. Assume that $A/G$ is smooth and $\dim(A^G)=0$. Then the triple $(A,\Theta,G)$ is a direct product of triples as in Proposition \ref{thm ppav irred}.
\end{corollary}

\begin{proof}
Since $\dim(A^G)=0$ and $A/G$ is smooth, by Theorem \ref{thm red irred} we have $A= A_1\times \cdots \times A_r$, $G=G_1\times\cdots\times G_r$ and $G_i$ acts irreducibly on $A_i$. Then, by \cite[Prop.~61]{Kani},
\[\NS(A)\cong\NS(A_1)\oplus\cdots\oplus\NS(A_r)\oplus\bigoplus_{i<j}\mathrm{Hom}(A_i,A_j),\]
where $G$ acts on each factor $\NS(A_i)$ by pullback and on $\mathrm{Hom}(A_i,A_j)$ by $\tau\cdot f=\tau f\tau^{-1}$. We note that there are no $G$-invariant elements in $\mathrm{Hom}(A_i,A_j)$ since $G_i$ acts irreducibly on $A_i$ and trivially on $A_j$.

Now, the coordinates of $\Theta$ with respect to the above decomposition are $G$-invariant. In particular, the coordinate of $\Theta$ in $\mathrm{Hom}(A_i,A_j)$ is 0. This implies that $\Theta$ splits as a sum of $G_i$-invariant principal polarizations $\Theta_i$ on each factor $A_i$. We can then apply Proposition \ref{thm ppav irred} to each factor.
\end{proof}

\subsection{The standard construction}\label{standard}
In the last section, we showed that when $\dim A^G=0$, then only products of Example \ref{ex1} can appear. In this section, we remove this hypothesis and present a way of constructing triples $(A,\Theta,G)$ such that $A/G$ is smooth and Example \ref{ex2} appears as a factor.\\

Consider the pair $(X,G)$ as in Example \ref{ex2}; in particular $X\cong E^g$ with $E$ an elliptic curve,  $G\cong S_{g+1}$ and the analytic representation is the standard representation of $S_{g+1}$. Define the following polarization on $X$:
\begin{equation}\label{eqn Xi0}
\Xi_{g}:=\Theta_g+\ker(\Sigma),
\end{equation}
where $\Sigma$ is the sum morphism $X\cong E^g\to E$ and $\Theta_g$ was defined in \eqref{eqn Theta0}. Note that $\Xi_g$ is a generator of the group $\NS(X)^{G}\cong\Z$ for $g\geq 2$ since it is primitive (cf.~for instance \cite[\S2.3]{Auff}) and in dimension 1 we have $\Xi_1=2\Theta_1$. \\

For a given polarization $\Xi$ on $X$, consider the group
\[K(\Xi):=\{x\in X\mid t_x^*\cal O_X(\Xi)\cong \cal O_X(\Xi)\},\]
where $t_x$ denotes translation by the element $x\in X$ (cf. \cite[Section 2.4]{BL}). 

\begin{lemma}\label{lemma type Ex b}
For $X$ as in Example \ref{ex2} and $\Xi_{g}$ as in \eqref{eqn Xi0}, we have
\[K(\Xi_{g})=X^G=\{(x,\ldots,x)\mid x\in E[g+1]\}\subset X.\]
In particular, $\Xi_g$ is of type $(1,\ldots,1,g+1)$.
\end{lemma}

\begin{proof}
Consider the lattice $(I_g\hspace{0.2cm}\tau I_g)\Z^{2g}$ of $E^g$ where $E=\C/(\Z+\tau\Z)$. With this lattice, the imaginary part of the first Chern class of $\Xi_g$ has matrix
\[M=\begin{pmatrix}
0 & A+I_g \\
-A-I_g & 0
\end{pmatrix},\]
where $A$ is the $g\times g$ matrix consisting of 1's in each coordinate. The group $K(\Xi_{g})$ corresponds to the set of all $x\in\Q^{2g}$, modulo $\Z^{2g}$, such that $x^tMy\in\Z^{2g}$ for all $y\in\Z^{2g}$, cf.~\cite[Lemma 2.4.5]{BL}. We obtain the group in the statement by a direct computation.
\end{proof}

Consider now a family of triples $(X_i,G_i,\Xi_{g_i})$ for $1\leq i
\leq r$, where $(X_i,G_i)$ is as in Example \ref{ex2} and $\Xi_{g_i}$ is defined in \eqref{eqn Xi0}. Define then the triple $(X,G,\Xi_X)$ as:
\begin{equation}\label{prodSn}X=\prod_{i=1}^rX_i,\quad G=\prod_{i=1}^rG_i,\quad \Xi_X=\Xi_{g_1}\boxtimes\cdots\boxtimes\Xi_{g_r},
\end{equation}
with the obvious action of $G$ on $X$. Let $(1,\ldots,1,d_1,\ldots,d_s)$ be the type of $\Xi_X$, let $(Y,\Xi_Y)$ be a polarized abelian variety of dimension $\geq s$ and of type $(1,\ldots,1,d_1,\ldots,d_s)$, and let $G$ act \emph{trivially} on it.\\

Starting from $(X,\Xi_X)$ and $(Y,\Xi_Y)$, we construct a principally polarized abelian variety $(A,\Theta)$ with a $G$-action that fixes the origin and preserves the class of $\Theta$ following \cite[\S9.2]{Debarre}. Note that $G$ acts trivially on both $K(\Xi_X)$ and $K(\Xi_Y)$.

Now, using \cite[Lemma 6.6.3]{BL}, one can construct an isomorphism $\epsilon:K(\Xi_X)\to K(\Xi_Y)$ that is antisymplectic with respect to the alternating forms induced by the respective polarizations (cf.~\cite[\S6.6]{BL}). Let $\Gamma$ be the graph of this isomorphism. Then, by \cite[Corollary 6.3.5]{BL}, there exists a unique principal polarization $\Theta$ on $A=(X\times Y)/\Gamma$ such that the pullback to $X\times Y$ gives $\Xi_X\boxtimes\Xi_Y$. Since $G$ acts trivially on $\Gamma$ and fixes the classes of $\Xi_X$ and $\Xi_Y$, we see that $G$ acts on $A$ and fixes the class of $\Theta$. Moreover, by Proposition \ref{prop And smooth}, the quotient $A/G$ is smooth since $X/G$ is by construction.

\begin{definition} 
We say that a principally polarized abelian variety with $G$-action $(A,\Theta,G)$ is \emph{standard} if it can be obtained via the construction here above.
\end{definition}

\subsection{The Main Theorem}

Having defined the standard construction in Section \ref{standard} and considering the results from Section \ref{sec dim A^G 0}, we are now ready to state our main result. It essentially states that the standard construction and $(E^g,\Theta_g)$ are the only direct factors of a principally polarized abelian variety admitting a smooth quotient.

\begin{theorem}\label{main thm}
Let $(A,\Theta)$ be a principally polarized abelian variety and let $G\subseteq \mathrm{Aut}(A,\Theta)$ be a non-trivial subgroup of automorphisms of $A$. Assume that $A/G$ is smooth. Then,
\[(A,\Theta)\cong\left(\prod_{i=1}^t(A_i,\Theta_{g_i})\right)\times (B,\Theta_B) \quad\text{and}\quad G\cong \left(\prod_{i=1}^tG_{i}\right)\times H,\]
where $(B,\Theta_B,H)$ is standard, each $(A_i,\Theta_{g_i},G_i)$ is as in Proposition \ref{thm ppav irred} and $G$ acts on each factor in the obvious way.

In particular $G$ is a direct product of symmetric groups and of groups of the form $(\Z/m\Z)^{g_i}\rtimes S_{g_i}$ with $m\in\{2,3,4,6\}$.
\end{theorem}

The idea of the proof is the following: we use the results given in Section \ref{sec prelim} in order to study and classify $G$-stable abelian subvarieties of $A$ according to whether they correspond to Examples \ref{ex1}, \ref{ex2} or \ref{ex3}. We will then use results by Debarre in order to prove that Example \ref{ex3} cannot occur if $A$ is principally polarized and that the subvarieties isomorphic to Example \ref{ex1} are already principally polarized and hence split as direct factors. Finally, we will prove that the remaining variety comes from the standard construction.\\

\begin{proof}[Proof of Theorem \ref{main thm}]
Let $(A,\Theta)$ be a principally polarized abelian variety and let $G\subseteq \aut(A,\Theta)$ be such that $A/G$ is smooth. Let $Y=(A^G)^0$, $X$ its complementary abelian subvariety with respect to $\Theta$, and let $\Xi_Y:=\Theta\cap Y$ and $\Xi_X:=\Theta\cap X$. Let $\Gamma=\ker(X\times Y\xrightarrow{+} A)$. The addition map is $G$-equivariant, and actually $G$ acts trivially on $\Gamma$. Moreover, by \cite[Prop.~9.1]{Debarre}, $\Gamma$ is the graph of an isomorphism $f:K(\Xi_X)\to K(\Xi_Y)$ that is antisymplectic with respect to the alternating forms induced by the respective polarizations. In particular, $G$ acts trivially on $K(\Xi_X)$ and $K(\Xi_Y)$.

By Theorem \ref{thm red irred} and Proposition \ref{prop And smooth}, $X/G$ is smooth and therefore 
\[X\cong X_1\times\cdots\times X_r\quad\text{and}\quad G\cong G_1\times\cdots\times G_r,\]
where $G_i$ acts on $X_i$ irreducibly and the quotient $X_i/G_i$ is smooth. Thus, by Theorem \ref{thm classification}, each pair $(X_i,G_i)$ corresponds to one of the Examples \ref{ex1}, \ref{ex2} or \ref{ex3}. Let $\Xi_{X_i}:=\Xi_X\cap X_i$ denote the restricted polarization. By following the proof of Corollary \ref{thm ppav inv of dim 0} (which does not use the fact that the polarization is principal until the very end), we deduce that
\[\Xi_X\equiv \Xi_{X_1}\boxtimes\cdots\boxtimes\Xi_{X_r},\]
and the numerical class of each $\Xi_{X_i}$ is fixed by $G$ (and thus $G_i$). An easy exercise gives us that
\[K(\Xi_X)=K(\Xi_{X_1})\oplus\cdots\oplus K(\Xi_{X_r}),\]
and $G$ (and thus $G_i$) acts trivially on each factor.

We will prove now that the polarizations $\Xi_{X_i}$ correspond to the ones we have defined above for Examples \ref{ex1} and \ref{ex2}. At the same time we will prove that Example \ref{ex3} cannot appear in this situation.

\begin{lemma}\label{lem reductions of Theta}
Let $g_i$ denote the dimension of $X_i$.
\begin{enumerate}
\item If $(X_i,G_i)$ is isomorphic to Example \ref{ex1} and $g_i\geq 2$, then $\Xi_{X_i}$ is the principal polarization $\Theta_{g_i}$ defined in \eqref{eqn Theta0}.
\item If $(X_i,G_i)$ is isomorphic to Example \ref{ex2} and $g_i\geq 2$, then $\Xi_{X_i}$ is the polarization $\Xi_{g_i}$ defined in \eqref{eqn Xi0}.
\item If $g_i=1$, then either $\Xi_{X_i}\equiv\Theta_1$ or $\Xi_{X_i}\equiv\Xi_1\equiv 2\Theta_1$. In this last case, $G_i\cong\Z/2\Z$.
\item None of the pairs $(X_i,G_i)$ can be isomorphic to Example \ref{ex3}.
\end{enumerate}
\end{lemma}

\begin{proof}
If $(X_i,G_i)$ is isomorphic to Example \ref{ex1}, then $X_i\cong E^{g_i}$ for some elliptic curve $E$ and $\NS(X_i)^{G_i}=\Z\cdot \Theta_{g_i}$. Therefore $\Xi_{X_i}\equiv m\Theta_{g_i}$ for some $m\in\mathbb{Z}_{>0}$. However,
\[K(m\Theta_{g_i})=X_i[m]=E[m]^{g_i}\]
which is only fixed by $G_i\cong (\Z/m\Z)^{g_i}\rtimes S_{g_i}$ in the case when $m=1$, hence $\Xi_{X_i}\equiv \Theta_{g_i}$. This proves (1).
 
If $g_i=1$, then the pair $(X_i,G_i)$ is always isomorphic to Example \ref{ex1}, so that the previous analysis still holds. However, in this case $X_i$ is an elliptic curve $E$ and $G_i\cong\Z/m\Z$, so that $E[m]$ can be fixed by $G_i$ if $m=2$ and $G_i=\{\pm 1\}\simeq\Z/2\Z$. This proves (3).\\

If $(X_i,G_i)$ is isomorphic to Example \ref{ex2} and $g_i\geq 2$, then $X_i\cong E^{g_i}$ for some elliptic curve $E$ and $\NS(X_i)^{G_i}=\Z\cdot \Xi_{g_i}$. Therefore $\Xi_{X_i}\equiv m\Xi_{g_i}$ for some $m\in\mathbb{Z}_{>0}$. However,
\[K(m\Xi_{g_i})=m^{-1}(K(\Xi_{g_i}))\supset X_i[m],\]
which is only fixed by $G_i\cong S_{g_i+1}$ in the case when $m=1$, hence $\Xi_{X_i}\equiv \Xi_{g_i}$. This proves (2).\\
 
If $(X_i,G_i)$ is isomorphic to Example \ref{ex3}, then with respect to the natural symplectic basis of the lattice $\Lambda:=\Z^2+i\Z^2$, we have that $\rho_r(G_i)$ is generated by the matrices
\[\left(\begin{array}{rrrr}
-1 & 1 & 0 & 1 \\
0 & 1 & 0 & 0 \\
0 & -1 & -1 & 1 \\
0 & 0 & 0 & 1
\end{array}\right), \left(\begin{array}{rrrr}
0 & -1 & -1 & 1 \\
0 & 0 & 0 & 1 \\
1 & -1 & 0 & -1 \\
0 & -1 & 0 & 0
\end{array}\right), \left(\begin{array}{rrrr}
-1 & 0 & 0 & 0 \\
-1 & 1 & 1 & 0 \\
0 & 0 & -1 & 0 \\
-1 & 0 & -1 & 1
\end{array}\right),\]
where $\rho_r:G\to\mathrm{GL}(\Lambda)$ denotes the rational representation of $G$. With respect to this basis the Riemann form is given by the matrix
\[J=\left(\begin{array}{cc}0&I\\-I&0\end{array}\right),\]
and a simple calculation with a computer program gives us that
\[c_1\left(\sum_{g\in G_i}g^*\mathcal{O}_{X_i}(\Theta_2)\right)=\sum_{g\in G_i}\rho_r(g)^tJ\rho_r(g)=\left(\begin{array}{rrrr}
0 & 16 & 32 & -16 \\
-16 & 0 & -16 & 32 \\
-32 & 16 & 0 & 16 \\
16 & -32 & -16 & 0
\end{array}\right).\]
Therefore $\sum_{g\in G_i}g^*\Theta_2\equiv 16\Xi_{\text{\ref{ex3}}}$ for a \emph{primitive} polarization $\Xi_{\text{\ref{ex3}}}$, and so $\Xi_{\text{\ref{ex3}}}$ generates $\NS(X_i)^{G_i}$. However, it is easy to see that
\[K(\Xi_{\text{\ref{ex3}}})=\langle(\tfrac{1+i}{2},0),(0,\tfrac{1+i}{2})\rangle,\]
which is not in the fixed locus of $G$. Moreover, $K(m\Xi_{\ref{ex3}})= m^{-1}(K(\Xi_{\text{\ref{ex3}}}))$ which is not invariant by $G$ either. Therefore by the previous analysis it is impossible for Example (c) to appear. This proves (4).
\end{proof}

Lemma \ref{lem reductions of Theta} tells us that the triples $(X_i,\Xi_i,G_i)$ are either Example \ref{ex1} with the polarization $\Theta_g$ or Example \ref{ex2} with the polarization $\Xi_g$. Indeed, the only case that ``escapes'' from this fact is when $g_i=1$ and $\Xi_{X_i}\equiv \Xi_{1}$. But since in this case we have $G_i\cong S_2$, we may indeed interpret the pair $(X_i,G_i)$ as Example \ref{ex2} in dimension 1 with the polarization $\Xi_{1}$.

Having said this, up to rearranging the factors, we can write
\[(X,\Xi_X)\cong \left(\prod_{i=1}^t(A_i,\Theta_{g_i})\right)\times\left(\prod_{j=1}^s(X_j,\Xi_{g_j})\right)\quad\text{and}\quad G\cong \left(\prod_{i=1}^tG_{i}\right)\times\left(\prod_{j=1}^sH_j\right), \]
where the pairs $(A_i,G_i)$ are isomorphic to Example \ref{ex1} and the pairs $(X_j,H_j)$ are isomorphic to Example \ref{ex2}. Note that since the pairs $(A_i,\Theta_{g_i})$ are principally polarized, they split as direct factors of $(A,\Theta)$ as well. Thus, we only need to prove that the remaining factor, which corresponds to the subvariety
\[B=Y+\sum_{j=1}^s X_j,\]
equipped with the principal polarization $\Theta_B:=\Theta\cap B$, is standard. This is an immediate application of \cite[Prop.~9.1]{Debarre} since $Y$ is the complementary abelian subvariety of $\sum_{j=1}^s X_j$ in $B$ with respect to $\Theta_B$ and the polarizations $\Xi_{g_j}$ are the ones used in the standard construction.
\end{proof}

We conclude this section with an immediate corollary to Theorem \ref{main thm}.

\begin{corollary}
Let $(A,\Theta)$ be a principally polarized abelian variety. Let $\aut(A,\Theta)$ be the group of automorphisms of $A$ preserving the numerical class of $\Theta$. Assume that $A/\aut(A,\Theta)$ is smooth. Then $(A,\Theta)$ is a polarized product of elliptic curves.
\end{corollary}

\begin{proof}
It suffices to note that the standard factor $(B,\Theta_B)$ from Theorem \ref{main thm} always ``has more'' automorphisms that preserve the polarization. Indeed, multiplication by $-1$ is an automorphism that does not fix the subvariety $Y$ (unless it is trivial) in this factor and by definition $Y$ was invariant by the group $H$ acting on $B$.
\end{proof}

\subsection{A moduli-theoretic interpretation}\label{sec moduli}
We note that Theorem \ref{main thm} has a moduli-theoretic interpretation. Indeed, the moduli space of triples $(E^m,\Xi_m,f)$ where $E$ is an elliptic curve, $\Xi_m$ is the polarization defined in (\ref{eqn Xi0}), and
\[f:K(\Xi_m)\to(\Z/(m+1)\Z)^2,\]
is a symplectic isomorphism is easily seen to be isomorphic to the modular curve $\cal X(m+1):=\mathbb{H}/\Gamma(m+1)$. As a consequence, the moduli space of all products of the form (\ref{prodSn}) along with a symplectic isomorphism
\[K(\Xi_X)\to\bigoplus_{i=1}^r(\Z/(g_i+1)\Z)^2,\]
is isomorphic to 
\[\cal X(g_1+1)\times\cdots\times \cal X(g_r+1).\]

Now, let $D=(1,\ldots,1,d_1,\ldots,d_s)$ be the type of $\Xi_X$ and note that this tuple depends only on the numbers $(g_1+1),\ldots,(g_r+1)$. Let $\cal{A}_{n}^D$ be the moduli space of triples $(Y,\Xi_Y,h)$ where $(Y,\Xi_Y)$ is a polarized abelian variety of dimension $n$ and type $D$, and 
\[h:K(\Xi_Y)\to\bigoplus_{i=1}^r(\Z/(g_i+1)\Z)^2,\]
is a symplectic isomorphism. Then the standard construction can be easily interpreted as a morphism of moduli spaces
\[\left(\prod_{j=1}^r\cal X(g_i+1)\right)\times\mathcal{A}_n^D\to\mathcal{A}_{g'},\]
where $g'=n+\sum_{i=1}^r g_i$. 

On the other hand, the factor in Theorem \ref{main thm} that does not come from the standard construction is just a product of principally polarized elliptic curves. Thus, its moduli space is simply a product of copies of $\mathcal{A}_1=\cal X(1)$. Therefore, we may interpret the set described in Theorem \ref{main thm} as the union of images of morphisms of moduli spaces
\[\Phi_{g_1,\ldots,g_s,n}:\left(\prod_{j=1}^s\cal X(g_i+1)\right)\times\mathcal{A}_n^D\to\mathcal{A}_g,\]
where $g_i,n\geq 0$, $\cal{A}_g$ is the moduli space of principally polarized abelian varieties of dimension $g$ and $g=n+z+\sum_{i=1}^sg_i$, where $z$ is the number of $i$'s such that $g_i=0$. The moduli-theoretic version of Theorem \ref{main thm} can henceforth be stated as follows:

\begin{theorem}\label{moduli}
A principally polarized abelian variety $(A,\Theta)$ admits a non-trivial subgroup $G\subseteq\aut(A,\Theta)$ that gives a smooth quotient $A/G$ if and only if it is in the image of one of the morphisms $\Phi_{g_1,\ldots,g_s,n}$.
\end{theorem}

We finish this section by studying the irreducibility of the theta divisor of a general element in the image of $\Phi_{g_1,\ldots,g_s,n}$. This will be useful in order to study smooth quotients of Jacobians. As it turns out, the only non-trivial case is the standard construction, which is studied in the following proposition.

\begin{proposition}\label{prop standard irred}
If $g_i>0$ for all $i$, a very general element in the image of $\Phi_{g_1,\ldots,g_s,n}$ is irreducible.
\end{proposition}

\begin{proof}
Let $(A,\Theta)$ be a very general element of the image of $\Phi_{g_1,\ldots,g_s,n}$, where specifically we mean that
\[A=(X_1\times\cdots\times X_s\times Y)/ \Gamma,\]
where $\Gamma$ is the graph of a certain anti-symplectic isomorphism, $\Hom(X_i,X_j)=0$ for $i\neq j$, $\Hom(X_i,Y)=0$ for all $i$ and $Y$ is simple. This implies, in particular, that
\begin{equation}\label{eqn End Q}
    \End_\Q(A)\cong\End_\Q(X_1)\oplus\cdots\oplus\End_\Q(X_s)\oplus\End_\Q(Y),
\end{equation}
where the subscript $\Q$ means that we are tensoring with $\Q$.

It is well-known that $(A,\Theta)$ is reducible if and only if there is a non-trivial abelian subvariety $T\subseteq A$ such that $\Theta\cap T$ is principal. Assume this is the case and let $S$ be its complementary abelian subvariety (which is also principally polarized). By \eqref{eqn End Q}, either $T$ or $S$ must contain $Y$ and the other must be contained in $\prod_{i=1}^sX_i$. Assume without loss of generality that this is the case for $T$. Then
\[(T,\Theta\cap T)=(T_1,\Theta\cap {T_1})\times\cdots\times (T_s,\Theta\cap {T_s}),\]
where $T_i=T\cap X_i$. Moreover, $\Theta\cap T_i$ is principal for every $i$ since $\Theta\cap T$ is. The following lemma tells us then that this is impossible, concluding the proof.
\end{proof}

\begin{lemma}
Let $X\subset E^n$ be an abelian subvariety of dimension $m>0$. Then $\Xi_n\cap X$ is \emph{not} a principal polarization.
\end{lemma}

\begin{proof}
Assume that $\Xi_n\cap X$ is principal and let us proceed by contradiction. Using the definition of $\Xi_n$ given in \eqref{eqn Xi0} and denoting $D_i:=\pi_i^*([0])$, we have that
\[m!=(\Xi_n\cap X)^m=\sum_{\Sigma k_i=m}\binom{m}{k_1,\ldots,k_{m+1}}D_1^{k_1}\cdots D_n^{k_n}(\ker\Sigma)^{k_{n+1}}\cdot X.\]
However, since each $D_i$ and $\ker\Sigma$ are abelian subvarieties, their self-intersection is trivial in the Chow ring modulo numerical equivalence. Hence, every $k_i$ can be taken to be equal to 0 or 1, i.e.
\begin{align*}
(\Xi_n\cap X)^m &=\underset{k_i\in \{0,1\}}{\sum_{\Sigma k_i=m}}\binom{m}{k_1,\ldots,k_{m+1}}D_1^{k_1}\cdots D_n^{k_n}(\ker\Sigma)^{k_{n+1}}\cdot X,\\
&=m! \underset{k_i\in \{0,1\}}{\sum_{\Sigma k_i=m}}D_1^{k_1}\cdots D_n^{k_n}(\ker\Sigma)^{k_{n+1}}\cdot X.
\end{align*}
But since $(\Xi_n\cap X)^m=m!$ and the intersection between $X$ and each summand is greater than or equal to 0 since the $D_i$'s and $\ker\Sigma$ are effective, we see that there is only one non-zero summand, which is equal to 1.

Let $F$ be an irreducible component of $\Xi_n$. Then $\Xi_n-F$ is ample and thus $(\Xi_n-F)^m\cdot X>0$. However, since there is only one non-zero summand in the equality above, we see that this number is equal to 0 as soon as $F$ is taken to appear in the non-zero summand, which is a contradiction.
\end{proof}

For the general case, note that if there exists $i$ such that $g_i=0$, then the theta divisor of every element of the image of $\Phi_{g_1,\ldots,g_s,n}$ is reducible. Putting this together with Proposition \ref{prop standard irred} we immediately get the following result.

\begin{theorem}\label{thm irred of Theta}
If the theta divisor of a principally polarized abelian variety with smooth quotient is irreducible, then it comes from the standard construction. Moreover, a very general element of this construction is irreducible.
\end{theorem}

\section{Smooth quotients of Jacobians}\label{jacobians}

In this section we prove Theorem \ref{thm jacob intro}. We start with a lemma on minimal morphisms. Following Kani (cf.~\cite{Kani2}), a cover $f:C\to E$ with $C$ a smooth curve and $E$ an elliptic curve is said to be \emph{minimal} if for every commutative diagram

\begin{equation}\label{eq diag minimal}
\xymatrix@R=0.8em@C=0em{
& C \ar[dl]_h \ar[dr]^f & \\
F \ar[rr] && E,
}
\end{equation}
where $F$ is an elliptic curve and $F\to E$ is an isogeny, we have $F=E$ and the isogeny is the identity. We have then the following Lemma:

\begin{lemma}\label{lem minimality}
Let $f:C\to E$ be a Galois cover of an elliptic curve $E$. Then $f$ is minimal.
\end{lemma}

\begin{proof}
Consider the commutative diagram \eqref{eq diag minimal}. Since $f$ is Galois, $F=C/H$ for some subgroup
$H\subseteq G$. Since the cover $F\to E$ is an unramified cover of elliptic curves, it is also Galois, and thus $H$ is normal in $G$. This implies that $h:C\to F$ is $G$-equivariant, and hence so is the morphism $h^*:F\to J_C$. Since $h^*F=f^*E\subset J_C^G$, $G$ acts trivially on $F$, and thus $H=G$ and $F=E$.
\end{proof}

We are now ready to prove Theorem \ref{thm jacob intro}.

\begin{proof}[Proof of Theorem \ref{thm jacob intro}]
It is obvious that the quotient $J_C/G$ is smooth if $g\leq 1$. Let us prove first then that either (2) or (3) implies that $J_C/G$ is smooth. Let $g'$ be the genus of $C':=C/G$ and let $R$ be the total ramification index of the covering $C\to C'$. We see then by Riemann-Hurwitz that $g'=1$ in case (2) and $g'=2$ in case (3), so $g'=g-1$ in both cases.

We get then that $J_{C'}$ is isogenous to an abelian subvariety $A_0$ of $J_C$ of dimension $g'=g-1$, which corresponds to the connected component of $J_C^G$ that contains 0. Let $P_G$ be the complementary abelian subvariety of $A_0$ with respect to the theta divisor of $J_C$, which has dimension 1. Since the action of $G=\Z/2\Z$ on $A_0$ is trivial and the theta divisor is $G$-invariant, we see that $G$ acts non-trivially on $P_G$ and thus $P_G /G\cong\bb P^1$ is smooth. Then by Proposition \ref{prop And smooth} we have that $J_C/G$ is smooth.\\

Assume now that $J_C/G$ is smooth with $g\geq 2$. We have to prove that $G\cong\mathbb{Z}/2\Z$ and that either (2) or (3) holds. The smoothness of $J_C/G$ at the image of $0$ tells us that $G$ is generated by pseudoreflections (i.e.~elements fixing pointwise a divisor passing through $0$) by the Chevalley-Shephard-Todd Theorem. Consider then a pseudoreflection $\sigma\in G$ and the subgroup $S\subset G$ generated by it. Then $J_C^S$ is a divisor and hence $J_{C/S}$ has dimension $g-1$. This implies that $C/S$ is a curve of genus $g-1$. A quick look at the Riemann-Hurwitz formula using $|S|\geq 2$ and $R\geq 0$ tells us that $g\leq 3$. We are left then with five possible cases:
\[(g,g')\in \{(3,2),(3,1),(3,0),(2,1),(2,0)\}.\]

In the case $(3,2)$, we obtain $|G|=2$ and $R=0$, which corresponds to case (3). In the case $(2,1)$, by \cite[Thm.~4.1]{Br} we have $|G|=2$ and the Riemann-Hurwitz formula yields $R=2$, which corresponds to case (2). We are left to prove then that the three other cases cannot give smooth quotients.\\

By Theorem \ref{thm irred of Theta}, we see that the pair $(J_C,\Theta_C)$ is standard since the theta divisor is an irreducible principal polarization. In particular $J_C$ is isogenous to a product $X\times Y$ with $Y$ a \emph{non-trivial} $G$-invariant abelian subvariety of $J_C$ (which \textit{a fortiori} corresponds to $f^*J_{C'}$) and $X$ a direct product of elliptic curves with $\dim(X^G)=0$ (which \textit{a fortiori} corresponds to the Prym subvariety of $J_C$ with respect to $G$). This discards immediately the cases $(2,0)$ and $(3,0)$, since then we have $\dim(J_C^G)=0$.

We are left then with the case $(3,1)$, where $X$ has dimension 2 and hence $G$ must be isomorphic to either $(\Z/2\Z)^2$ or $S_3$ by the standard construction. Here, $C'$ is an elliptic curve $E$ and thus $J_E=E$. Lemma \ref{lem minimality} tells us that the cover $f:C\to C'=E$ is minimal.

Assume that $G\cong S_3$, let $H$ denote the index 2 normal subgroup of $G$ and consider the quotient $C''=C/H$. Then the genus $g''$ of $C''$ must be 2 since it has to be $<3$ and if it was 1 we would contradict the minimality of $f:C\to E$. We get then that $C\to C''$ is a Galois cover with Galois group $H$ and $|H|> 2$. This contradicts the Riemann-Hurwitz formula.

Assume now that $G=(\Z/2\Z)^2$. Since $f$ is minimal and of degree $|G|$, \cite[Cor.~12.1.4]{BL} tells us that
\[|f^*E\cap X|=|G|^2=16.\]
But since the action of $G$ on $E$ is trivial and there exists $g\in G$ acting as $-1$ on $X$, we know that $f^*E\cap X\subset X^G\subset X[2]$ and hence $f^*E\cap X$ is 2-torsion. Since $f^*E$ is just an elliptic curve, $|f^*E\cap X|\leq |(f^*E)[2]|=4$, which yields a contradiction.
\end{proof}

Theorem \ref{thm jacob intro} immediately gives us the following interesting corollary.

\begin{corollary}\label{cor false}
Let $g_1,\ldots,g_r,n\in\mathbb{Z}_{\geq0}$ and let $z$ be the number of $i$'s such that $g_i=0$. If $n+z+\sum_{i=1}^rg_i\geq4$, then the image of $\Phi_{g_1,\ldots,g_r,n}$ is disjoint from the Jacobian locus.
\end{corollary}

\begin{rem}\label{rem incomplete}
Using well-known results by Broughton \cite{Br}, we see that the moduli space of genus 3 \'etale double covers of genus 2 curves is a connected 3-dimensional subvariety of $\mathcal{M}_3$, and  the moduli space of genus 2 double covers of elliptic curves is a connected 2-dimensional subvariety of $\mathcal{M}_2$.
\end{rem}

\appendix

\section{Corrigendum}

The proof of Theorem \ref{thm jacob intro} uses Lemma \ref{lem minimality}. This lemma is false, and the only case of the proof where it was used, namely the case $(g,g')=(3,1)$, has to be revisited. It turns out that this case contains smooth quotients that we had overlooked. Moreover, Corollary \ref{cor false} does not hold as stated.

In this section we give the corrected statement of Theorem \ref{thm jacob intro} together with a proof of the case $(3,1)$ and we explain what goes wrong with Corollary \ref{cor false}. The results of Sections 1--3 are not affected.

We keep the notation from above.

\section*{The mistake in Lemma 4.1}

Lemma \ref{lem minimality} claims that every Galois cover $f\colon C\to E$ of an elliptic curve $E$ is minimal. This is false: a non-trivial isogeny of elliptic curves is a Galois cover that is not minimal, and examples where $C$ has genus 3 are given by Lemma \ref{lem S3 structure} below. The error in the proof lies in the sentence ``Since $h^*F=f^*E\subset J_C^G$, $G$ acts trivially on $F$''. Indeed, the group $G/H$ acts on the elliptic curve $F=C/H$ as the Galois group of the \'etale cover $F\to E$, that is, by translations. Since translations act trivially on $J_F=\Pic^0(F)$, the $G$-equivariance of $h^*\colon J_F\to J_C$ gives no information on $H$.

\section*{The corrected statement}

Lemma \ref{lem minimality} was used only once, in the case $(g,g')=(3,1)$ of the proof of Theorem \ref{thm jacob intro}. The correct statement is the following.

\begin{theorem}[Theorem \ref{thm jacob intro} corrected]
Let $X$ be a smooth projective curve of genus $g$ and let $G$ be a (non-trivial) group of automorphisms of $X$. Then $J_X/G$ is smooth if and only if one of the following holds:
\begin{enumerate}
\item $g\leq 1$;
\item $g=2$, $G\cong\mathbb{Z}/2\Z$  and $X\to X/G$ ramifies at two points;
\item $g=3$, $G\cong\mathbb{Z}/2\Z$  and $X\to X/G$ is \'etale;
\item $g=3$, $G\cong S_3$ and $X/G$ is an elliptic curve.
\end{enumerate}
\end{theorem}

In case (4), the cover $X\to X/G$ is branched over exactly one point (Lemma \ref{lem S3 structure}). In this case, the quotient $J_X/G$ is a $\bb P^2$-bundle over an elliptic curve (Proposition \ref{prop S3 smooth}).

\section*{Smooth quotients in the case \texorpdfstring{$(3,1)$}{(3,1)}}

\begin{lemma}\label{lem S3 structure}
Let $C$ be a curve of genus 3 and let $G\subseteq\aut(C)$ be a subgroup isomorphic to $S_3$ such that $E:=C/G$ is an elliptic curve. Then:
\begin{enumerate}
\item $F:=C/A_3$ is an elliptic curve, the quotient morphism $h\colon C\to F$ is totally ramified at exactly two points, namely the fixed points of $A_3$, and the induced double cover $F\to E$ is \'etale;
\item the elements of order 2 of $G$ act freely on $C$, and the quotient morphism $f\colon C\to E$ is branched over exactly one point.
\end{enumerate}
In particular, $f$ factors through the isogeny $F\to E$ and it is not minimal.
\end{lemma}

\begin{proof}
Let $\rho$ be a generator of $A_3$ and let $r$ be the number of fixed points of $\rho$. Since $F$ covers $E$, we have $g(F)\geq 1$, and the Riemann-Hurwitz formula for $h$ reads $4=3(2g(F)-2)+2r$. Hence $g(F)=1$ and $r=2$, and the Riemann-Hurwitz formula for $F\to E$ shows that this double cover is \'etale. Now, the Riemann-Hurwitz formula for $f$ gives
\[4=\sum_{P\in C}(|G_P|-1),\]
where $G_P$ denotes the stabilizer of $P$ in $G$. The two fixed points of $\rho$ contribute at least $2\cdot 2=4$ to this sum, with equality if and only if their stabilizer is exactly $A_3$. Hence all other points have trivial stabilizer, so that the elements of order 2 act freely on $C$. Finally, since $A_3$ is normal in $G$, the two fixed points of $\rho$ are permuted by $G$, and since their stabilizer $A_3$ has index 2, they form a single orbit. Thus $f$ is branched over exactly one point.
\end{proof}

In what follows, we denote by $V$ the standard representation of $S_3$ (that is, its unique irreducible representation of dimension 2), by $\mathbf{1}$ the trivial representation and by $A_3\subset S_3$ the subgroup of order 3. Otherwise we keep the notation of Lemma \ref{lem S3 structure}.

\begin{proposition}\label{prop S3 smooth}
Let $C$ be a curve of genus 3 and let $G\subseteq\aut(C)$ be a subgroup isomorphic to $S_3$ such that $C/G$ is an elliptic curve. Let $Y:=(J_C^G)^0$ and let $X$ be its complementary abeliane subvariety with respect to $\Theta_C$. Then:
\begin{itemize}
\item[(1)] the quotient $J_C/G$ is smooth;
\item[(2)] the pair $(X,G)$ is isomorphic to Example \ref{ex2} with $g=2$, so that $X/G\cong\bb P^2$, and the triple $(J_C,\Theta_C,G)$ is standard, the polarizations $\Theta_C\cap X$ and $\Theta_C\cap Y$ being of type $(1,3)$ and $(3)$ respectively;
\item[(3)] $J_C/G$ is a $\bb P^2$-bundle over an elliptic curve isomorphic to $C/A_3$.
\end{itemize}
\end{proposition}

\begin{proof}
For every subgroup $H\subseteq G$, the dimension of $H^0(C,\omega_C)^H$ is the genus of $C/H$. By Lemma \ref{lem S3 structure}, both $H^0(C,\omega_C)^G$ and $H^0(C,\omega_C)^{A_3}$ are one-dimensional. This implies that $H^0(C,\omega_C)\cong\mathbf{1}\oplus W$ with $W^{A_3}=0$. Since the only representation of $S_3$ with nontrivial action of $A_3$ is $V$, we must have $W=V$. Taking the dual, we see that the analytic representation of $G$ on $T_0J_C$ is also isomorphic to $\mathbf{1}\oplus V$. The subspace $T_0Y$ is the trivial summand and, since $X$ is $G$-stable (as $\Theta_C$ is $G$-invariant), $G$ acts on $T_0X$ via $V$.

The endomorphism $1+\rho+\rho^2$ of $X$ vanishes on $V=T_0X$, hence it is zero. Thus, on $X$, we have $1+\rho=-\rho^2$ and
\[(1-\rho)(1-\rho^2)=3-(1+\rho+\rho^2)=3,\qquad 1-\rho^2=(1-\rho)(1+\rho)=-(1-\rho)\rho^2.\]
Taking degrees, we get $\deg(1-\rho)^2=\deg(3_X)=3^4$, so that $1-\rho$ is an isogeny of $X$ of degree 9. Hence $X^\rho=\ker(1-\rho)$ has order 9.

The group $G$ acts on $F=C/A_3$ through $G/A_3$, which is the Galois group of the \'etale cover $F\to E$. Hence $G$ acts on $F$ by translations and thus trivially on $J_F$, so that $h^*J_F\subseteq J_C^G$. Since $\dim J_C^G=g(C/G)=1$, we get $Y=h^*J_F$. By Lemma \ref{lem S3 structure}, $h$ is totally ramified at some point, so that it does not factor through any non-trivial \'etale cover, and thus $h^*\colon J_F\to J_C$ is injective by \cite[Prop.~11.4.3]{BL}. Moreover, $(h^*)^*\Theta_C\equiv 3\theta_F$ by \cite[Lemma 12.3.1]{BL}, where $\theta_F$ denotes the principal polarization of $J_F$. Therefore $\Xi_Y:=\Theta_C\cap Y$ is three times a principal polarization and $K(\Xi_Y)=Y[3]$ has order 9. Finally, as in the proof of Theorem \ref{main thm}, we have $X\cap Y=K(\Xi_X)=K(\Xi_Y)$, where $\Xi_X:=\Theta_C\cap X$ (cf.~\cite[Cor.~12.1.5]{BL}). Hence $|X\cap Y|=9$, and $\Xi_X$ is of type $(1,3)$.

Now, since $G$ acts trivially on $Y$, we have $X\cap Y\subseteq X^G\subseteq X^\rho$ and thus these three groups must coincide. In particular, every point of $X$ that is fixed by $\rho$ is fixed by $G$. Hence the stabilizer $G_x$ of a point $x\in X$ is either trivial, or of order 2, or equal to $G$. Since $G$ acts on $T_xX=T_0X$ via $V$, the elements of order 2 act on it as reflections, and they generate $G$. By the Chevalley-Shephard-Todd Theorem, $X/G$ is smooth at the image of every point $x\in X$. Thus $X/G$ is smooth, and so is $J_C/G$ by Proposition \ref{prop And smooth}. This proves (1).

Since $X/G$ is smooth and $V$ is irreducible, Theorem \ref{thm classification} tells us that $X/G\cong\bb P^2$ and that the pair $(X,G)$ is as in one of the Examples \ref{ex1}, \ref{ex2} or \ref{ex3}. Since $|G|=6$, it must be Example \ref{ex2} with $g=2$. Moreover, since $\Theta_C$ is irreducible, the triple $(J_C,\Theta_C,G)$ is standard by Theorem \ref{thm irred of Theta}. This ends the proof of (2).

Finally, as in the proof of Theorem \ref{main thm}, we have $J_C=(X\times Y)/\Gamma$, where $\Gamma=\{(x,-x)\mid x\in X\cap Y\}$. The group $\Gamma$ acts on $X$ by translations by points of $X\cap Y\subseteq X^G$, which commute with the action of $G$, and freely on $Y$ by translations by points of $X\cap Y=Y[3]$. Since $G$ acts trivially on $Y$, we get
\[J_C/G\cong(X/G\times Y)/\Gamma,\]
which is a $\bb P^2$-bundle over the elliptic curve $Y/Y[3]\cong Y\cong J_F\cong F$. This proves (3).
\end{proof}

\begin{rem}
One can prove that the pairs $(C,G)$ appearing in Proposition \ref{prop S3 smooth} are precisely given by the family of hyperelliptic curves given by the affine equation
\[y^2=x(x^6+ax^3+1),\]
parametrized by $a\in\C\smallsetminus\{\pm2\}$. The automorphisms spanning $S_3$ are $\rho,\tau\in \aut(C_a)$, of respective orders $3$ and $2$, defined by
\[\rho(x,y)=(\omega x,\omega^2y)\quad\text{and}\quad\tau(x,y)=\left(\frac{1}{x},-\frac{y}{x^4}\right),\]
where $\omega\in\C$ denotes a primitive cube root of unity. In particular, complementing Remark \ref{rem incomplete}, the moduli of these curves is irreducible of dimension one, and its image in $\cal M_3$ is an irreducible curve contained in the hyperelliptic locus.
\end{rem}

\section*{Proof of Theorem \ref{thm jacob intro}}

We now explain how to modify the proof of Theorem \ref{thm jacob intro}, where the curve is denoted by $C$ and $g'$ denotes the genus of $C':=C/G$. In the first part of the proof, the fact that each of the conditions (2) and (3) implies that $J_C/G$ is smooth is correctly proved, and condition (4) implies that $J_C/G$ is smooth by Proposition \ref{prop S3 smooth}. In the second part of the proof, where $J_C/G$ is assumed to be smooth with $g\geq 2$, the reduction to the five cases
\[(g,g')\in\{(3,2),(3,1),(3,0),(2,1),(2,0)\},\]
as well as the treatment of the cases $(3,2)$, $(2,1)$, $(3,0)$ and $(2,0)$, remain unchanged. The last three paragraphs of the proof, which deal with the case $(3,1)$, should be replaced by the following ones.

\medskip

We are left then with the case $(3,1)$, where $X$ has dimension 2 and hence $G$ must be isomorphic to either $(\Z/2\Z)^2$ or $S_3$ by the standard construction. Here, $C'$ is an elliptic curve $E$ and $f^*E=(J_C^G)^0$ is the complementary abelian subvariety of $X$. If $G\cong S_3$, then we are in case (4).

Assume now that $G\cong(\Z/2\Z)^2$. Then, by the standard construction, $X=X_1\times X_2$ with $X_1,X_2$ elliptic curves, and $G=H_1\times H_2$, where $H_i\cong\Z/2\Z$ acts on $X_i$ as $\{\pm1\}$ and trivially on the other factor. Moreover, $\Xi_X:=\Theta_C\cap X$ satisfies $\Xi_X\equiv\Xi_{X_1}\boxtimes\Xi_{X_2}$, where $\Xi_{X_i}:=\Theta_C\cap X_i$. By Lemma \ref{lem reductions of Theta}(3), we have either $\Xi_{X_i}\equiv\Theta_1$ or $\Xi_{X_i}\equiv2\Theta_1$. If $\Xi_{X_i}\equiv\Theta_1$ for some $i$, then $\Theta_C$ restricts to a principal polarization on the abelian subvariety $X_i$ of $J_C$, which implies that $(J_C,\Theta_C)$ is reducible, a contradiction. Hence $\Xi_{X_i}\equiv 2\Theta_1$ for $i=1,2$, and since $f^*E\cap X=K(\Xi_X)$ (see the proof of Theorem \ref{main thm}), we get
\[|f^*E\cap X|=|K(\Xi_X)|=|X_1[2]\times X_2[2]|=16.\]
But since the action of $G$ on $f^*E$ is trivial and there exists $g\in G$ acting as $-1$ on $X$, we know that $f^*E\cap X\subset X^G\subset X[2]$ and hence $f^*E\cap X$ is 2-torsion. Since $f^*E$ is just an elliptic curve, $|f^*E\cap X|\leq|(f^*E)[2]|=4$, which yields a contradiction. \qed

\section*{Corollary \ref{cor false}}

Corollary \ref{cor false} was deduced from Theorem \ref{thm jacob intro}. This deduction is not valid: Theorem \ref{thm jacob intro} only deals with subgroups of $\aut(C)$, while, by Theorem \ref{moduli}, the image of $\Phi_{g_1,\ldots,g_r,n}$ consists of principally polarized abelian varieties $(A,\Theta)$ admitting a smooth quotient by \emph{some} subgroup of $\aut(A,\Theta)$, and $\aut(J_C,\Theta_C)=\aut(C)\times\{\pm1\}$ when $C$ is not hyperelliptic. In fact, Corollary \ref{cor false} is false, as the following example shows.

\begin{example}\label{ex bielliptic}
Let $C$ be a bielliptic curve of genus $g\geq 2$, that is, a curve with an involution $s$ such that $E:=C/\langle s\rangle$ is an elliptic curve. The automorphism $s$ acts on $T_0J_C$ with eigenvalue $1$ with multiplicity $g(E)=1$ and eigenvalue $-1$ with multiplicity $g-1$. Hence $\sigma:=-s\in\aut(J_C,\Theta_C)$ acts on $T_0J_C$ as a reflection. Since the stabilizer in $\langle\sigma\rangle\cong\Z/2\Z$ of every point of $J_C$ is either trivial or $\langle\sigma\rangle$, the quotient $J_C/\langle\sigma\rangle$ is smooth by the Chevalley-Shephard-Todd Theorem; by Proposition \ref{prop And smooth}, it is actually a $\bb P^1$-bundle over an abelian variety of dimension $g-1$. By Theorem \ref{moduli}, $(J_C,\Theta_C)$ lies in the image of one of the morphisms $\Phi_{g_1,\ldots,g_s,n}$ with $n+z+\sum_{i=1}^s g_i=g$ (namely, of $\Phi_{1,g-1}$). Since there exist bielliptic curves of every genus, this contradicts Corollary \ref{cor false} for every $g\geq 4$.
\end{example}

Corollary \ref{cor false} should therefore be withdrawn. The correct statement in this direction is the fact, contained in Theorem \ref{thm jacob intro}, that $J_C/G$ is singular for every curve $C$ of genus $g\geq 4$ and every non-trivial subgroup $G\subseteq\aut(C)$.

\subsection*{Acknowledgements} We are grateful to Rub\'i Rodr\'iguez for pointing out the mistake in Lemma \ref{lem minimality}.


\begin{thebibliography}{ABCD00}

\bibitem[Auf17]{Auff} R.~Auffarth. \textit{A note on Galois embeddings of abelian varieties}. Manuscripta Math.~154(3-4), 2017, pp.~279--284.

\bibitem[ALA20]{ALA} R.~Auffarth, G.~Lucchini Arteche. \textit{Smooth quotients of abelian varieties by finite groups}. Ann. Sc. Norm. Super. Pisa Cl. Sci. (5) XXI (2020), 673--694.

\bibitem[ALAQ18]{ALAQ} R.~Auffarth, G.~Lucchini Arteche, P.~Quezada. \textit{Smooth quotients of abelian surfaces by finite groups}. Preprint. arXiv:1809.05405.

\bibitem[Bro90]{Br} S.~A.~Broughton. \textit{Classifying finite group actions on surfaces of low genus}. J.~Pure Appl.~Algebra 69, 1990; pp.~233--270.

\bibitem[BL04]{BL} Ch.~Birkenhake, H.~Lange. \textit{Complex abelian varieties.}
Second edition. Grundlehren der Mathematischen Wissenschaften 302. \textit{Springer-Verlag, Berlin}, 2004.

\bibitem[Deb88]{Debarre} O.~Debarre. \textit{Sur les vari\'et\'es ab\'eliennes dont le diviseur theta est singulier en codimension 3.} Duke~Math.~J.~57, 1988, no. 1, pp.~221--273.

\bibitem[Kan94]{Kani2} E.~Kani. \textit{Elliptic curves on abelian surfaces.} Manuscripta Math.~84, 1994, pp.~199--223.

\bibitem[Kan16]{Kani} E.~Kani. \textit{The moduli spaces of Jacobians isomorphic to a product of two elliptic curves.} Collect.~Math.~67, 2016, pp.~21--54. 


\end{thebibliography}
\end{document}